\documentclass[a4paper,10pt]{amsart}

\usepackage{inputenc}
\usepackage{color}
\usepackage{graphicx}
\usepackage{times,mathptm}
\usepackage{amsfonts,amscd,amssymb,amsmath}
\usepackage{xypic}
\usepackage{latexsym}

\newtheorem{proposition}{Proposition}[section]
\newtheorem{lemma}[proposition]{Lemma}
\newtheorem{corollary}[proposition]{Corollary}
\newtheorem{theorem}[proposition]{Theorem}

\theoremstyle{definition}
\newtheorem{definition}[proposition]{Definition}
\newtheorem{example}[proposition]{Example}

\theoremstyle{remark}

\newcommand{\lemlabel}[1]{\label{lem:#1}}
\newcommand{\lemref}[1]{Lemma~\ref{lem:#1}}

\newcommand{\thelabel}[1]{\label{the:#1}}
\newcommand{\theref}[1]{Theorem~\ref{the:#1}}

\def\ot{\otimes}

\def\-{{\rm-}}

\def\dim{{\rm dim}}

\def\ov{\overline}
\def\el{\rm el}
\def\rk{\rm rk}
\def\Proj{\rm Proj}
\def\Gal{\rm Gal}

\def\mapright#1{\smash{\mathop{\longrightarrow}\limits^{#1}}}

\def\hookmapright#1{\smash{\mathop{\hookrightarrow}\limits^{#1}}}

\newcommand{\veq}{\mathrel{\rotatebox{90}{$=$}}}

\makeindex

\begin{document}
\title{Clifford theory for glider representations}
\author[F. Caenepeel]{Frederik Caenepeel}
\address{Department of Mathematics, University of Antwerp, Antwerp, Belgium}
\email{Frederik.Caenepeel@uantwerpen.be}
\author[F. Van Oystaeyen]{Fred Van Oystaeyen}
\address{Department of Mathematics, University of Antwerp, Antwerp, Belgium}
\email{Fred.Vanoystaeyen@uantwerpen.be}
\subjclass[2010]{20C05}
\keywords{Clifford theory, fragment}
\begin{abstract}
Classical Clifford theory studies the decomposition of simple $G$-modules into simple $H$-modules for some normal subgroup $H \triangleleft G$. In this paper we deal with chains of normal subgroups $1 \triangleleft G_1 \triangleleft \cdots \triangleleft G_d =G$, which allow to consider fragments and in particular glider representations. These are given by a descending chain of vector spaces over some field $K$ and relate different representations of the groups appearing in the chain. Picking some normal subgroup $H \triangleleft G$ one obtains a normal subchain and one can construct an induced fragment structure. Moreover, a notion of irreducibility of fragments is introduced, which completes the list of ingredients to perform a Clifford theory.
\end{abstract}

\thanks{The first author is Aspirant PhD Fellow of FWO}

\maketitle
\section{Introduction}
In his original paper \cite{Cl}, A. H. Clifford elucidates the behaviour of simple $G$-representations $V$ when considered as $H$-representations for some normal subgroup $H \triangleleft G$, where $G$ is some finite group. It appears that there are two possiblities; the induced $H$-representation $V_H$ is either itself simple, or decomposes into irreducible components all of the same degree, which are moreover conjugate relative to $G$ to one another. By conjugate representations $V$ and $W$ we mean that there is some $g \in G$ such that $h \cdot V = g^{-1}hg \cdot W$ for all $h \in H$. Clifford constructs a so-called decomposition group $H \subset G' \subset G$, which fully determines the representation $V$. This raises the opposite question of embedding a simple $H$-representation into some given simple $G$-representation, and a crucial role is again played by the same decomposition group $G'$. The answer is given by
\begin{theorem}\thelabel{classic}
A given irreducible $H$-representation $U$ can be embedded in an irreducible $G$-representation $V$ if and only if 
\begin{enumerate}
\item the subgroup $G' \subset G$ consisting of elements $g \in G$ such that the conjugate $H$-representation $g \cdot U$ is equivalent to $U$, is of finite index in $G$;
\item $U$ can be embedded in a simple $G'$-representation $U'$.
\end{enumerate}
If the ground field $K$ is algebraically closed, the latter condition is equivalent to the existence of a representation of finite degree of the group algebra $K[G'/H]$ with multiplication
$$u_S u_T = \alpha(s,t)^{-1}u_{ST},$$
where $s, t$ are representatives of $S,T \in G'/H$ and where $\alpha: G' \times G' \to K$ is a factor-set associated to $U$.
\end{theorem}
Glider representations of a finite group $G$ are defined for some chain of normal subgroups $ 1 \triangleleft G_1 \triangleleft \cdots \triangleleft G_{d-1} \triangleleft G_d = G$. These are substructures of a $KG$-module $\Omega$, $K$ a base field, say $M$, with given descending chain $M = M_0 \supset M_1 \supset \cdots \supset M_d \supset \cdots$ such that $KG_i \cdot M_j \subset M_{j-i}$ for $i \leq j$ where the operation of $KG_i$ on $M_j$ is induced by the $KG$-structure of $\Omega$. Glider representations may be understood as intermediate structures relating $KG$-modules to $K$-vectorspaces via the chain of groups considered in $G$. These glider representations are thus not given by modules but by fragments in the sense of \cite{CVo}, \cite{EVO}, \cite{NVo1} and they provide information about relations between representations of the groups $G_i$ appearing in the chain. Since fragments for a given chain do not form a nice Abelian category, the theory of glider representations is essentially different from the classical representation theory of groups and we have to develop this theory almost from scratch. In this paper we generalize Clifford theory for induced glider representations.
We consider a chain of normal subgroups $ 1 \triangleleft G_1 \triangleleft \cdots \triangleleft G_{d-1} \triangleleft G_d$ and pick a normal subgroup $H \triangleleft G$. By putting $H_i = H \cap G_i$, we obtain an embedding of filtered group algebras $FKH \to FKG$, where $F_iKH = KH_i, F_iKG = KG_i$. It is obvious that an $FKG$-fragment $M$ can be seen as an $FKH$-fragment, which corresponds to the usual forgetful functor $U: G-{\rm rep} \to H-{\rm rep}$. Under some additional conditions, we provide a construction of an induced fragment. That is, to an $FKH$-fragment $N$ we associate an $FKG$-fragment $N^G$. These constructions allow us to perform a Clifford theory after recalling (and changing somewhat) the notion of irreducibility for fragments from \cite{EVO}. On the way we also provide a few general facts on fragments over finite algebra filtrations. 

\section{Preliminaries}
We begin by recalling the definition of a fragment $M$ over a filtered ring $FR$ from \cite{CVo}.
\begin{definition}
Let $FR$ be a positive filtration with subring $S = F_0R$. A (left) $FR$-fragment $M$ is a (left) $S$-module together with a descending chain of subgroups
$$M_0 = M \supseteq M_1 \supseteq \cdots \supseteq M_i \supseteq \cdots$$
satisfying the following properties\\
$\bf{f_1}$. For every $i \in \mathbb{N}$ there exists an $S$-module $M \supseteq M_i^* \supseteq M_i$ and there is given an operation of $F_iR$ on this $M_i^*$ by $\varphi_i: F_iR \times M_i^* \to M,~(\lambda,m) \mapsto \lambda.m$, satisfying $\lambda.(m + n) = \lambda.m + \lambda.n, 1.m = m, (\lambda + \delta).m = \lambda.m + \delta.m$ for $\lambda,\delta \in F_iR$ and $m,n \in M_i^*$.\\

$\bf{f_2}$. For every $i$ and $j \leq i$ we have a commutative diagram
$$\xymatrix{ M & M_{i-j} \ar@{_{(}->}[l]^i \ar@{^{(}->}[r]_i & M\\
F_iR \times M_i \ar[u]^{\varphi_i} & F_jR \times M_i \ar@{_{(}->}[l]^{i_F} \ar[u] \ar@{^{(}->}[r]_{i_M} & F_jR \times M_j \ar[u]_{\varphi_j}}$$

$\bf{f_3}$. For every $i,j,\mu$ such that $F_iRF_jR \subset F_\mu R$ we have $F_jRM_\mu \subset M_i^* \cap M_{\mu - j}$.
Moreover, the following diagram is commutative
$$\xymatrix{
F_iR \times F_jR \times M_\mu \ar[d]_{F_iR \times \varphi_\mu} \ar[rr]^{m \times M_\mu} && F_\mu R \times M_\mu \ar[d]_{\varphi_\mu} \\
F_iR \times M_{\mu - j} \ar[rr]^{\ov{\varphi_i}} && M},$$
in which $\ov{\varphi_i}$ stands for the action of $F_iR$ on $M_i^*$ and $m$ is the multiplication of $R$. Observe that the left vertical arrow is defined, since $1 \in F_0R$ implies that $F_jR \subset F_\mu R$. 
\end{definition}
For an $FR$ fragment structure on $M$, the chain $M \supseteq M_1^* \supseteq M_2^* \supseteq \cdots$ obviously also yields an $FR$-fragment. If the fragmented scalar multiplications $\phi_i : F_iR \times M_i \to M$ are induced from an $R$-module $\Omega$, that is, when $M \subset \Omega$, we call $M$ a glider representation. In this case we have that $M_i^* = \{ m \in M,~F_iRm \subset M\}$. If for all $i$ we moreover have that $M_i^* = M_i$, we say that $M$ is natural.\\

In case $FR$ is given as a ring filtration, i.e. each $F_iR$ is a subring of $R$, the first part of the fragment condition $\bf{f_3}$ is equivalent to $F_\mu RM_\mu \subset M_\mu ^*$ for all $\mu$. Indeed, $F_iR F_jR \subset F_\mu R$ implies that $F_iR, F_jR \subset F_\mu R$ since $1 \in F_0R$ by definition. Then we have that $M_\mu^* \subset M_i^*,M_j^*$, so $F_jRM_\mu \subset F_\mu R M_\mu \subset M_\mu^* \subset M_i^*$. Conversely, from $F_jRM_\mu \subset M_i^*$ for all $i$ and $j$ such that $F_iRF_jR \subset F_\mu R$ we have in particular $F_\mu RF_\mu R \subset F_\mu R$ hence $F_\mu RM_\mu \subset M_\mu^*$. We observe that a natural fragment over a ring filtration is a chain of $F_\mu R$-modules.\\

If the filtration on $R$ is exhaustive, i.e. $\cup_n F_nR = R$, then it follows that $\cap_i M_i = B(M)$ is an $R$-module. We call $B(M)$ the body of the fragment $M$.
\begin{definition}
An $S$-submodule $N$ of an $FR$-fragment $M$ is said to be a subfragment if there is a chain $N = N_0 \supseteq N_1 \supseteq \cdots \supseteq N_i \supseteq \cdots$ such that $N_i \subseteq M_i$ and the action of $F_iR$ on $M_i$ induces an action on $N_i$ making $N$ into an $FR$-fragment.
\end{definition}
 In order to apply a Clifford theory in the fragment or glider setting, we introduce the notion of an irreducible fragment. Since a fragment is given by a descending chain of $F_0R$-modules for some filtered ring $FR$, there are some trivial ways of defining subfragments. 
\begin{definition}
Let $FR$ be a filtered ring and $M$ an $FR$-fragment. A subfragment $N$ of $M$ is said to be trivial if either\\
$T_1$. There is a $k \geq 0$ such that $N_k = B(N)$ but $M_k \neq B(M)$.\\
$T_2$. There is a $k \geq 0$ such that $N_k = 0$ but $M_k \neq 0$.\\
$T_3$. There exists monotone increasing map $\alpha: \mathbb{N} \to \mathbb{N}$ such that $N_k = M_{\alpha(k)}$ and $\alpha(k) - l \geq \alpha(k - l)$ for all $l \leq k$.
\end{definition}
A subfragment $N$ of $M$ is strict if $N_k = N \cap M_k$ for all $k \geq 0$.
\begin{definition}
A fragment $M$ is said to be irreducible if all of its subfragments are trivial. $M$ is said te be weakly irreducible if all strict subfragments are trivial.
\end{definition}
If there exists an $e \in \mathbb{N}$ such that $M_e \neq B(M)$, but $M_{e+1} = B(M)$, then we say that $M$ has essential lenght $\el(M) =e$.\\

Since we will be working with group algebra filtrations, we assume that $FA$ is a finite algebra filtration, with $F_dA = A$.
\begin{lemma}\lemlabel{el}
Let $M$ be a weakly irreducible $FA$-fragment such that $M \neq B(M)$, then there is an $e \in \mathbb{N}$ such that $M_e \neq B(M)$ and $e$ is maximal as such. For this $e$, we have that $M_i = F_{e-i}AM_e$, for $0 \leq i \leq e$.
\end{lemma}
\begin{proof}
Suppose that $M_m \neq B(M)$ for all $m \in \mathbb{N}$. Then the chain $F_nAM_n \supseteq M_1 \cap F_nAM_n \supseteq \cdots \supseteq M_n \cap F_nAM_n = M_n \supseteq \cdots$ is a strict subfragment of $M$ with body $B(M)$ ($\neq M_n$ for every $n$). Thus either $M_n = M_{\alpha(n)}$ with $\alpha(n) \geq n, M_m = M_{\alpha(m)}$ for every $m \geq n$. This yields $M_n = B(M)$, a contradiction. Hence $M = F_nAM_n = F_{n+1}AM_{n+1} = \cdots = F_{n+i}AM_{n+i} = \cdots$, for $i \in \mathbb{N}$. But $M = F_{d+ n+i}AM_{d+ n+i} = F_dAM_{d + n+i} \subset M_{n+i}$, thus $M = M_1 = \cdots = M_i = \cdots$, or $M = B(M)$, a contradiction. Consequently such an $e$ exists and the subfragment of $M$ given by
$$F_eAM_e \supseteq F_{e-1}AM_e \supseteq \ldots \supseteq M_e \supset B(M) \supseteq \ldots$$
is trivial, from which the last statement follows.
\end{proof}
If $M$ is such that $\el(M)= e > d$, then $M = F_{e}AM_e = F_dAM_e \subset M_{e -d} \subset M$. Therefore we don't lose essential information in considering $M= M_{e-d} \supseteq M_{e -d + 1} \supseteq \ldots$. If the essential length of this new fragment is still strictly larger than $d$, we can shift again until we reach $\el(M) = e \leq d$. In \cite{EVO}, it is shown that killing the body $B(M)$ preserves the essential length and irreducibility. Therefore, we may restrict the study to fragments with zero body and essential length $ e \leq d$. Such a fragment consists of an $F_eA$-module $M$ with descending chain of $F_{e-i}A$-modules $M_{i}$. Observe that this is opposite to natural fragments, where $M_i$ is an $F_iA$-module. In \cite{EVO} it is shown that for finite semisimple algebra filtrations $FA$, every finitely generated natural fragment $M$ is a direct sum of weakly irreducible strict subfragments. If $M$ is a glider representation with strict subfragment $N$, one can naturalise $N^* \subset M^*$. However, $N^*$ is no longer strict in general as the following example shows.

\begin{example}
Let $G$ be a (finite) group and consider the filtration $K \subset KG$, $K$ some field. Let $W$ be a $G$-representation and $V \subset W$ a $K$-subspace of dimension $> 1$ and which is not a $G$-rep. Consider $a \in V$, then 
$$\begin{array}{ccc}
W & \supset & Ka\\
\cup && \cup \\
V & \supset & Ka
\end{array}$$
is a strict subfragment. Since $W$ is a $G$-rep, we obtain 
$$\begin{array}{cccc}
W^*: & W & \supseteq & W\\
&\cup && \cup \\
V^*: & V & \supset & \{v \in V | Gv \subseteq V\} \neq V
\end{array}$$
which is not strict.
\end{example}
We do have a similar decomposition for glider representations.
\begin{lemma}
Let $FA$ be a finite semisimple algebra filtration with $F_dA = A$ and $M$ an $FA$ glider representation. Then every strict subfragment $N$ of $M$ is a direct summand.
\end{lemma}
\begin{proof}
Consider $N^{(M^*)} \subseteq M^*$ the strict subfragment in $M^*$, i.e. $N^{(M^*)}_i = N \cap M_i^*$. By \cite[Lemma 4.1]{EVO} there exists a strict subfragment $L \subseteq M^*$ such that $N^{(M^*)}_i \oplus L_i = M_i^*$ as $F_iA$-modules for all $i$. As $K$-vectorspaces we obtain for all $i$ that 
$$M_i = (N^{(M^*)}_i \cap M_i) \oplus (L_i \cap M_i) = N_i \oplus (M_i \cap L_i).$$
In particular, for $i = 0$, we have $L_0 \cap M_0 = L \cap M = L$. Moreover, since $L$ is strict in $M^*$, we have that $L \cap M_i = L_i \cap M_i$ and $L$ is strict in $M$. We arrive at $M = N \oplus L$, with $L$ a strict subfragment. 
\end{proof}
\begin{proposition}
A finitely generated glider representation $M$ is a direct sum of weakly irreducible subfragments.
\end{proposition}
\begin{proof}
Since $M$ is finitely generated, there exists a weakly irreducible strict subfragment $N \subseteq M$. In view of the foregoing, $M = N \oplus L$ for some strict subfragment $L$. The same can be applied to $L$, which is finitely generated and the result follows.
\end{proof}
Now look at an irreducible glider representation $N$ given by $N \supseteq N_1 \supseteq \ldots \supseteq N_d \supset 0 \ldots$, where $N_d$ is a 1-dimensional $K$-space ($A$ $K$-algebra). If $u$ is a unit of $F_dA$ then we can also look at $F_dAuN_d = F_dAN_d = N \supseteq F_{d-1}AuN_d \supseteq \ldots \supseteq uN_d \supset 0 \ldots$ and we assume $u$ is chosen such that $uN_d \neq N_d$. The latter is an irreducible fragment and we obtain a non-irreducible glider representation $N \supseteq N_1 + F_{d-1}AuN_d \supseteq \ldots \supseteq N_d \oplus  uN_d \supset 0 \ldots$ containing at least the two irreducible fragments we used in the construction. These two are not disjoint so the sum is not a direct sum. Nevertheless we have $N_d \cap N_du  = 0$. Therefore, we say that a sum of fragments $E + F$ is direct if for some $ i \leq \el(E),\el(F)$ we have that $E_i$ is disjoint from $F_i$.  A motivation for this definition is the information given by the chain of the fragment. A direct sum on every level would be too set- and module-theoretic. We recall the following decomposition result
\begin{theorem}\cite[Theorem 4.7]{EVO}\\
Let $FA$ be a finite semisimple algebra filtration on a finite dimensional $K$-algebra $A$ and let $M$ be a finitely generated $FA$-fragment with $B(M) = 0$ and essential length $\el(M) = d$. Then $M$ is the fragment direct sum of irreducible fragments.
\end{theorem}

\section{Induction of fragments for filtration extensions of groups}
Let $G$ be a finite group and $H \triangleleft G$ some normal subgroup. Then the short exact sequence $$ 1 \to H \to G \mapright{\pi} G/H \to 1$$ may be viewed as defining an extension of $H$ by $G/H$ via the construction of a set map $\sigma: G/H \to G$ such that $\pi \circ \sigma(\ov{g}) = \ov{g}$ for $\ov{g} \in G/H$. Fixing $\sigma$ defines $g = \sigma(\ov{g})h$ for a unique $h \in H$. In particular, $\sigma(\ov{g_1})\sigma(\ov{g_2}) = \sigma(\ov{g_1g_2})h(\ov{g_1},\ov{g_2})$ defines a map 
$$h(-, - ): G/H \times G/H \to H.$$
From $(g_1g_2)g = g_1(g_2g)$, it follows that $h$ is a 2-cocycle, i.e. it satisfies the following condition
\begin{equation} \label{cocycle}
h(\ov{g_1g_2},\ov{g})h(\ov{g_1},\ov{g_2})^{\rho_{\sigma(\ov{g})}} = h(\ov{g_1},\ov{g_2g})h(\ov{g_2},\ov{g}),
\end{equation}
where $(-)^{\rho_{\sigma(\ov{g})}}$ denotes the conjugation by $\sigma(\ov{g})$.\\

Now look at a group algebra filtration of $KG$, $K$ some field, given by a chain of normal subgroups $ 1 \triangleleft G_1 \triangleleft \cdots \triangleleft G_{d-1} \triangleleft G_d$, that is, $F_nKG = KG_n$ for $ 0 \leq n \leq d$. For $H$ a normal subgroup of $G$, put $H_i = G_i \cap H$. We obtain a group algebra filtration of $KH$. By the normality of all subgroups, we have a commutative diagram
$$\xymatrix{
G_1/H_1 \ar@{^{(}->}[r]  \ar[d]^{{\rm iso}} & G_2/H_2 \ar@{^{(}->}[r]  \ar[d]^{{\rm iso}} & \cdots \ar@{^{(}->}[r] & G/H \ar[d]^{id}\\
G_1H/H \ar@{^{(}->}[r]  & G_2H/H  \ar@{^{(}->}[r]  & \cdots \ar@{^{(}->}[r]  & G/H}$$
which allows us to consider an ascending transversal set $1 \subseteq T_1 \subseteq T_2 \subseteq \cdots \subseteq T_i \subseteq \cdots \subseteq T$, where $T_i$ is a set of (right) coset representations of $H_i$ in $G_i$. This transversal set gives rise to a 2-cocycle $h$ as before. 
\begin{definition}
A 2-cocycle $h: G/H \times G/H \to H$ is said to be filtered if $h(-,\ov{g}): G/H \to H$ is restricting to $(G/H)_i \to H_i$, where $(G/H)_i = G_i/H_i$, for every $i$ and $g \in G$.
\end{definition}

\begin{example}
If $\pi: G \to G/H$ admits a group section $\sigma$, then $G = HN$ for some subgroup $N$ of $G$ such that $N \cap H = \{1\}$. Then choosing $N$ as transversal set $T$ yields $h(-,-) = 1$, which is filtered for any chain of subgroups of $G$.
\end{example}

\begin{example}
Let $K \hookmapright{} L$ be a Galois extension with finite Galois group $G = \Gal(L/K)$ and fix a chain of normal subgroups 
$$ \{1\} \triangleleft G_1 \triangleleft \cdots \triangleleft G_d = G.$$
The Galois correspondence yields a field filtration $FL$
$$ K = K_0 \subset K_1 \subset \cdots \subset K_d = L,$$
where $K_i = L^{G_{d-i}}$. Let $A \supseteq A_1 \supseteq \cdots \supseteq A_d = L \supseteq 0 \cdots$ be a natural $FL$-fragment, i.e. $A_i$ is a $K_i$-algebra for every $0 \leq i \leq d$. Assume moreover that every $A_i$ is an Azumaya algebra over $K_i$. For $\sigma \in G_i \setminus G_{i-1}$, we obtain by the Skolem-Noether theorem a unit $u_\sigma \in A^\times$ such that for all $x \in L$
$$\sigma(x) = u_\sigma^{-1} x u_\sigma.$$
Since $L^{G_i} = K_{d-i}$, we get that $u_\sigma \in Z_A(K_{d-i}) = Z_A(Z_A(A_{d-i}))$. If $A_{d-1}$ is a simple algebra, then the Centralizer Theorem yields that $u_\sigma \in A_{d-i}^\times$. Moreover, since $\sigma \notin G_{i-1}$, there exists $y \in K_{d-i +1} \setminus K_{d-i}$ such that $\sigma(y) \neq y$. This shows that $u_\sigma \notin A_{d-i +1}$. If $\dim_L(A_i) = \dim_{K_i}(L) = |G_{d-i}|$, then
$$ A_i \cong \bigoplus_{\sigma \in G_{d-i}} L u_\sigma.$$
Since the Brauer group is isomorphic to the second Galois cohomology group, the isomorphism for $i = 0$ is given by a 2-cocycle $f: G \times G \to L^\times$. If $f$ appears to be filtered, then we have in particular that $f$ is restricting to
$$f: G_{d -i} \times G_{d- i} \to K_{d-i}^\times.$$
This implies that the Azumaya $K_i$-algebra $A_i$ has a subalgebra isomorphic to
$$ \bigoplus_{\sigma \in G_{d-i}} K_{d-i} u_\sigma \subset \bigoplus_{\sigma \in G_{d-i}} L u_\sigma \cong A_{i}.$$
Since $\dim_K(K_{d-i}) = |\Gal(K_{d-i}/K)| =|G_{d-i}|$, this subalgebra is Azumaya over $K$. Hence, if an Azumaya algebra is determined by a filtered 2-cocycle for some fixed chain of normal subgroups, we obtain a chain of Azumaya algebras over the corresponding fixed fields, which all have subalgebras that are Azumaya over $K$
$$\begin{array}{ccccccc}
\bigoplus_{\sigma \in G} L u_\sigma & \supseteq & \bigoplus_{\sigma \in G_{d-1}} Lu_\sigma & \supseteq & \cdots & \supseteq & Lu_{1}\\
\cup && \cup &&&& \cup \\
\bigoplus_{\sigma \in G} L u_\sigma & \supseteq & \bigoplus_{\sigma \in G_{d-1}} K_{d-1} u_\sigma & \supseteq & \cdots & \supseteq & Ku_{1}\\
\end{array}$$
\end{example}

In case $h(-,-)$ is filtered and in case all $H_i$ are normal in $G$, e.g. when all $G_i$ are normal in $G$, we can extend an $FKH$-fragment $N$ into an $FKG$-fragment as follows. Put  $ M = K[G/H] \ot_K N$ for the $K$-space $\oplus_{\ov{g} \in G/H} K\ov{g} \ot N$. Similarly, for every $j$ we put $M_j = K[G/H] \ot_K N_j$. In this way, we obtain a descending chain of $K$-spaces
$$ M = M_0 \supseteq M_1 \supseteq \cdots \supseteq M_d \supseteq 0 \supseteq \cdots$$
In order to define a $KG_j$ multiplication on $M_j$, it will be enough to define $g_j(\ov{g} \ot n_j$) and extend this $K$-bilinearly. We let $\sigma: G \to T$ be as before (choice of transversal) with \newline $\sigma: G_j \to T_j$ for all $j$.
Define for $g_1 \in G_j:~ g_1 \cdot (\ov{g} \ot n_j) = \ov{g_1g} \ot h(\ov{g_1},\ov{g})t^{-1}h_1t n_j$, where $g_1  = t_1h_1, h_1 \in H_j, g = th,~ t_1 = \sigma(g_1)$ and $t = \sigma(g)$, $n_j \in N_j$. If $n_j \in N_{l+j}$ in the foregoing, then since $h(\ov{g_1},\ov{g}) \in H_j$ and $t^{-1}h_1t \in H_j$ we have that $g_1 \cdot (\ov{g} \ot n_j) \in \ov{g_1g} \ot N_l \subset M_l$. First we verify that for $g_1 \in G_j, g_2 \in G_i$ and $n \in N_{\max\{i,j\}}$, we have that
$$ g_2g_1(\ov{g} \ot n) = g_2(g_1(\ov{g} \ot n)) \in \ov{g_2g_1g} \ot N.$$
So consider $g_2 \in G_i, g_2 = t_2h_2$ with $t_2 = \sigma(g_2), h_2 \in H_2,  g_1 = t_1h_1$ with $t_1 = \sigma(g_1)$ and $h_1 \in H_j$ and $n \in N_{l}$ with $l \geq \max\{i,j\}$. Then 
$$g_2g_1 = t_2h_2t_1h_1 = t_2t_1(t_1^{-1}h_2t_1)h_1 = \sigma(g_2g_1)h(\ov{g_2},\ov{g_1})(t_1^{-1}h_2t_1)h_1.$$
 Therefore
$$g_2g_1(\ov{g} \ot n) = \ov{g_2g_1g} \ot h(\ov{g_2g_1},\ov{g})(t^{-1}h(\ov{g_2},\ov{g_1})t)(t^{-1}t_1^{-1}h_2t_1t)(t^{-1}h_1t)n.$$
On the other hand, we have
\begin{eqnarray*}
g_2(g_1(\ov{g} \ot m)) &=& g_2[\ov{g_1g} \ot h(\ov{g_1},\ov{g})(t^{-1}h_1t)n]\\
&=&  \ov{g_2g_1g} \ot h(\ov{g_2},\ov{g_1g})(\sigma(g_1g)^{-1}h_2\sigma(g_1g))h(\ov{g_1},\ov{g})(t^{-1}h_1t)n.
\end{eqnarray*}
Both expressions are equal since $h(\ov{g_2},\ov{g_1g})(\sigma(g_1g)^{-1}h_2\sigma(g_1g))h(\ov{g_1},\ov{g})$ is equal to
\begin{eqnarray*}
& & (h(\ov{g_2},\ov{g_1g})h(\ov{g_1},\ov{g}))(h(\ov{g_1},\ov{g})^{-1}\sigma(g_1g)^{-1}h_2\sigma(g_1g)h(\ov{g_1},\ov{g}))\\
&=& h(\ov{g_2g_1},\ov{g})(t^{-1}h(\ov{g_2},\ov{g_1})t) (t^{-1}t_1^{-1} h_2 t_1t).
\end{eqnarray*}
The third fragment conditon $\bf{f_3}$ will follow from the following proposition.
\begin{proposition}
If $N$ is a (glider, resp. natural) $KH$-fragment, then $M = N^G$ is a (glider resp. natural) $KG$-fragment. 
\end{proposition}
\begin{proof}
We have to establish that $KG_\mu (KT \ot N_\mu) \subset M_\mu^*$, where $M_\mu = KT \ot N_\mu$.\\
Take $ m \in M_\mu,~ m= \oplus_{t \in T} \lambda_t t \ot n_t$ with $n_t \in N_\mu,~ \lambda_t \in K$. For $g_\mu \in KG_\mu$ we have $g_\mu \cdot m = \oplus_t \lambda_t \ov{t_\mu t} \ot h(t_\mu, t)t^{-1}h_\mu t n_t$ where $\ov{t_\mu t}$ is the representative for $t_\mu t$ in the transversal $T$ and $h(t_\mu, t)t^{-1}h_\mu t \in H_\mu$. From $KH_\mu N_\mu \subset N_\mu ^*$ we see that for every $g_\mu' \in G_\mu$ we have $g_\mu' g_\mu m \subset KT \ot N = M$. Hence $g_\mu m \in M_\mu^*$ or $KG_\mu M_\mu \subset M_\mu^*$.\\
Suppose that $N \subset \Omega$ is a $KH$-module inducing the operations, then $M = N^G \subset KT \ot \Omega = \Omega^G$ is a $KG$-module inducing the operations of $G_i$ on $M_i = KT \ot N_i$, and we see that $M$ is a glider representation. If $N$ is natural, look at $m \in M,~ m = \oplus_{t \in T} \lambda_tt \ot n_t,~ \lambda_t \in K,~ n_t \in N$ and assume $KG_i m \subset M$, i.e. $m \in M_i^*$. If $g_i = t_ih_i$ in $G_i$, then $g_im = \oplus \lambda_t \ov{t_i t} \ot h(t_i,t) t^{-1}h_itn_t$. Since $h(t_i,t) \in H_i$ for all $t$ and $H_i$ is normal in $G$, we obtain $H_i n_t \subset N$ or $n_t \in N_i^* = N_i$ as $N$ is natural. Consequently $m \in KT \ot N_i = M_i$ and $M$ is natural.
\end{proof}
We call $M \supseteq \cdots \supseteq M_j \supseteq \cdots$ the induced fragment of $N$ and denote it by $N^G$.
\begin{lemma}
For $M = N^G$, we have that $M_\mu^* = KT \ot N_\mu^*$ for all $\mu$.
\end{lemma}
\begin{proof} 
Let $m \in M_\mu^*$, then for $g_\mu = t_\mu h_\mu,~ g_\mu m = \oplus_t \lambda_t \ov{t_\mu t} \ot h(t_\mu,t) t^{-1}h_\mu t n_t$, where \newline$m = \oplus_{t \in T} \lambda_tt \ot n_t,~ \lambda_t \in K,~ n_t \in N$. So $g_\mu m = \oplus_t \lambda_t \ov{t_\mu t} \ot h_\mu' n_t$ for some $h_\mu' \in H_\mu$.
Since $g_\mu m \in M$ it follows that $g_\mu m = \oplus_t b_t t \ot a_t$ with $b_t \in K$ and $a_t \in N$. Since $KT \ot N$ is a direct sum $\oplus Kt \ot N$ ($ \cong \oplus_{t\in T} N$), we have, up to some permutation of $T$, say $\sigma$, that $Kh_\mu' n_t = Ka_{\sigma(t)} \subset N$. By the choice of $h_\mu$ in $g_\mu$ we can obtain every $h_\mu^1$ for $h_\mu'$, hence $H_\mu n_t \subset N$, or $n_t \in N_\mu^*$. The other inclusion is trivial.
\end{proof}
\begin{corollary}
For an $FKH$ glider representation $N$, we have that $N$ is natural if and only if $N^G$ is natural.
\end{corollary}

\section{Clifford theory of group algebra fragments}


We continue with the group algebra filtrations $FKH \hookmapright{} FKG$. First, we discuss the going-up direction, that is, we see what happens to the induced fragment of an irreducible $FKH$-fragment. In this section, we do everything for fragments of essential length $d$ and zero body. In fact, everything is analogous for smaller essential lengths. So consider $N = N_0 \supseteq N_1 \supseteq \cdots \supseteq N_{d-1} \supseteq N_d \supset 0 \ldots$ an irreducible $FKH$-fragment of essential length $d$. Irreducibility implies that $N_d$ is one dimensional and by \lemref{el} we know that for any $i$, $N_{d-i} = KH_{i}N_d$, i.e. $N_{d-i}$ is a $KH_i$-module. Remark that we do not know whether the $N_i$ are simple $KH_{d-i}$-modules (they are semisimple for suitable $K$). It is not even the case that simplicity of $N_i$ as $KH_{d-i}$-module implies simplicity of $N_{i+1}$ as $KH_{d-i-1}$-module. A thorough study of irreducible glider representations for chains of group algebras is work in progress.\\

From now on, we assume that ${\rm char}(K) =0$. All group algebras over $K$ for finite groups are therefore semisimple. Define $M$ to be the induced $FKG$-fragment, that is
$$ M = K[G/H] \ot N.$$
Then on degree $d$ we have a direct sum (of $K$-spaces) $M_d = T \ot Ka$ where $N_d = Ka$.  Any $t \in T$ generates an irreducible $FKH$-fragment
$$ KGt \ot a \supseteq KG_{d-1}t \ot a \supseteq \cdots \supseteq Kt \ot a.$$
Since $M_{d-i} = K[G/H] \ot N_{d-i} = KT \ot KH_iN_d = KG_iT \ot N_d = KG_iM_d,$
where the third equality follows from
$$\ov{t} \ot h_in_d = th_it^{-1} \cdot \ov{t} \ot n_d, {\rm ~and~} g_i \ov{t} \ot n_d = \ov{t'} \ot h(\ov{g_i},\ov{t})t^{-1}h_itn_d \quad (g_i =sh_i),$$
we decomposed $M$ into a fragment direct sum of $|T|$ irreducible $FKH$-fragments all of essential length $d$. So we have
\begin{theorem}
Let $N$ be an irreducible $FKH$-fragment of essential length $d$. Then the induced fragment $M = K[G/H] \ot N$ decomposes into a fragment direct sum of $[G:H]$ irreducible $FKH$-fragments of essential length $d$.
\end{theorem}
As a corollary, we obtain a Mackey decomposition theorem. Indeed, suppose that $E \triangleleft G$ is another normal subgroup. We have a commutative diagram
$$\xymatrix{ 
1 \ar[r] & H \ar[r] & G \ar[r]^\pi & G/H \ar@/_1pc/[l]_\sigma \ar[r] &1\\
1 \ar[r] &E \cap H \ar@{^{(}->}[u] \ar[r] & E \ar[r]^\pi  \ar@{^{(}->}[u] & \frac{E}{E \cap H}  \ar@/^1pc/[l]^\sigma\ar@{^{(}->}[u] \ar[r] & 1}$$
and by putting $E_i = G_i \cap E$, we obtain two additional group algebra filtrations $FKE$ and $FK(E\cap H)$. By the normality condition, we can begin by fixing an ascending transversal set $1 \subseteq S_1 \subseteq \cdots \subseteq S_d = S$, with $S_i$ a set of right coset representations of $E_i \cap H_i$ in $E_i$. Consider now a subset $U_i \subset G_i$ such that $\{E_itH_i~|~ t \in U_i\}$ is a complete set of double coset representatives. Then $T_i = \{st~|~ s \in S_i\}$ is a complete set of representatives for $H_i$ in $G_i$ and $S_i \subseteq T_i$. In this way, we obtain an ascending transversal set $1 \subseteq T_1 \subseteq \ldots \subseteq T_d = T$, with $T_i$ a set of right coset representations of $H_i$ in $G_i$ and the associated 2-cocylce $h: G/H \times G/H \to H$ is restricting to $h: E/(E \cap H) \times E/(E \cap H) \to E \cap H$. Moreover, for $s \in S$ and $t \in U_d = U$, we have that $h(\ov{s},\ov{t}) = 1$.  In the case that $h$ is filtered, we can induce an $FKH$-fragment $N = N_0 \supseteq N_1 \supseteq \cdots \supseteq N_d \supset 0 \cdots$ to an $FKG$-fragment 
$$ M = N^G = K[T] \ot N.$$
\begin{theorem}
Let $H, E$ be normal subgroups of a finite group $G$, with fixed ascending chain of normal subgroups $$ 1 \triangleleft G_1 \triangleleft \cdots \triangleleft G_d =G.$$
Let $N$ be an $FKH$-fragment with $FKA$ the induced group algebra filtration on $A = H,E,H \cap E$. Then the induced fragment restricted as an $FKE$-fragment $(M^G)_K$ is the fragment direct sum 
$$ (N^G)_E \cong \bigoplus_{t \in U} \big[(t \ot N)_{H\cap E}\big]^E.$$
\end{theorem}
\begin{proof}
By construction $N^G = \bigoplus_{t \in T} t \ot N$, and for fixed $t \in T$, the descending chain
$$t \ot N \supseteq t \ot N_1 \supseteq \cdots \supseteq t \ot N_0$$
is easily seen to be an $FK(H\cap E)$-fragment. For $t \in U$ define
$$\varphi: \bigoplus_{s \in S} s \ot t \ot N \to \bigoplus_{s \in S} st \ot N,~ s \ot t \ot n \mapsto st \ot m.$$
Let $k = s_1z_1 \in K_i$ with $s_1 \in S_i, z_i \in E_i \cap H_i$ and $n \in N_i$. On the one hand we have
\begin{eqnarray*}
k \cdot \varphi(s \ot t \ot n) &=& k \cdot st \ot n = \ov{kst} \ot h(\ov{k},\ov{st})t^{-1}s^{-1}z_1stn\\
&=& \ov{kst} \ot h(\ov{k},\ov{st})h(\ov{s},\ov{t})t^{-1}s^{-1}z_1stn.
\end{eqnarray*}
On the other hand, we calculate
\begin{eqnarray*}
\varphi(k \cdot s \ot t \ot n) &=& \varphi(\ov{ks} \ot h(\ov{k},\ov{s})s^{-1}z_1s \cdot (t \ot n))\\
&=& \varphi(\ov{ks} \ot t \ot t^{-1}h(\ov{k},\ov{s})s^{-1}z_1st n)\\
&=& \ov{kst} \ot h(\ov{ks},\ov{t})t^{-1}h(\ov{k},\ov{s})s^{-1}z_1st n).
\end{eqnarray*}
Both expressions are equal by the 2-cocycle condition \eqref{cocycle}, hence $\varphi$ is a morphism of $FKE$-fragments. The map is easily seen to be surjective and as $K$-spaces the domain and codomain have the same dimension, so we have an isomorphism of $FKE$-fragments. The result now follows.
\end{proof}
Let us now consider the going down direction of the Clifford theory. So suppose that $M$ is an irreducible $FKG$-fragment and consider $M = M_H$ as an $FKH$-fragment. By irreducibility, $M_d = Ka$ is a one-dimensional $K$-vectorspace. To begin, we observe that $KH_1a \subset M_{d-1}$ is a $KH_1$-submodule, which decomposes into simple $KH_1$-modules
$$KH_1a = S^0_1 \oplus \cdots \oplus S^0_{e_0},$$
since $KH_1$ is semisimple. Let $S_1^1,\ldots, S_{e_1}^1$ be simple $KH_1$-modules such that 
$$M^H_{d-1} = KH_1a \oplus S^1_1 \oplus \cdots \oplus S^1_{e_1}.$$
Subsequently, any $1 \leq i \leq e_1$ gives a (trivial) $FKH$-subfragment
$$ KH_{d-1}S^1_i \supseteq \ldots \supseteq KH_2S^1_i \supseteq S^1_i \supseteq 0 $$
of essential length $d-1$. Next, we find simple $KH_2$-modules $S^2_1,\ldots,S^2_{e_2}$ such that
$$M^H_{d-2} = KH_2(S^0_1 \oplus \cdots \oplus S^0_{e_0}) + KH_2(S^1_1 \oplus \cdots \oplus S^1_{e_1}) \oplus S^2_1 \oplus \cdots \oplus S^2_{e_2}.$$
Observe that the first sum no longer needs to be direct, testifying to the higher complexity of fragment structures. For every $1 \leq i \leq e_2$ we again obtain a (trivial) $FKH$-subfragment
$$KH_{d-2}S^2_i \supseteq \ldots \supseteq S^2_i \supseteq 0 \supseteq 0$$
of essential length $d-2$. Proceeding in this way, we arrive at a decomposition

$$M^H = M_0^H = KH(S^0_1 \oplus \cdots \oplus S^0_{e_0}) + KH(S^1_1 \oplus \cdots \oplus S^2_{e_2}) + \cdots $$
$$+ ~KH(S^{d-1}_1 \oplus \cdots \oplus S^{d-1}_{e_{d-1}}) \oplus S^d_1 \oplus \cdots \oplus S^d_{e_d}.
$$
Summarizing, we decomposed $M^H$ as a fragment direct sum of ``irreducible" fragments with lowest non-zero part $S^i_j$ ($0 \leq i \leq d, 1 \leq j \leq e_i$) and of essential length $d-i$.\\

Inspired by the classical Clifford theory, we can say something more. In our construction, we viewed a simple $KH_i$-module $S = S^i_j$ inside the $KH_{i+1}$-module $KH_{i+1}S$. Since we do not know whether the latter is irreducible, this is not entirely the classical embedding problem. Nonetheless, we can mimic the construction of \cite{Cl} and use a different approach that will lead to so-called decomposition groups. We will see however, that these decomposition groups will lie between $H_i$ and $H_{i+1}$ and not between $H_i$ and $G_i$ as in the classical case.\\

In the decomposition of $M^H_{d-1}$ into simple $H_1$-modules above, we fix some $S = S^1_1$. For any $h \in H_2 \setminus H_1$, $hS$ and $S$ are conjugate $KH_1$-modules relative to $H_2$. Moreover, $hS$ is also simple. If $\forall h \in H_2$, we would have that $hS = S$, $S$ would be a simple $KH_2$-module and thus $KH_2S = S$. Otherwise, we find a finite number of elements $h_2,\ldots, h_r \in H_2 \setminus H_1$ such that
$$ S \oplus h_2S \oplus \cdots \oplus h_rS = KH_2S.$$ 
After regrouping all equivalent $H_1$-modules together, we get
$$ R_1 \oplus R_2 \oplus \cdots \oplus R_m = KH_2S,$$
where the $R_i$ are the sum of equivalent modules. We define the subgroup $H'_{2,1} \subset H_2$ of elements which leave $R_1$ invariant. By definition, $KH_2S$ is transitive, so all the spaces $R_i$ must have the same dimension. Moreover, the $KH'_{2,1}$-module $R_1$ generates $KH_2S$. However, $R_1$ need not be a simple $H'_{2,1}$-module, since we do not have that $KH_2S$ is simple (cf. \cite[Observations before section 3]{Cl}). E.g. for $\mathbb{Z}_2 \subset \mathbb{Z}_4$, the two-dimensional $\mathbb{Z}_4$-representation $V$ defined by
$$ 1 \mapsto \begin{pmatrix} i & 0 \\ 0 & i \end{pmatrix}$$
has $S = \mathbb{C}e_1$ has simple $\mathbb{Z}_2$-representation and $V = S \oplus 1\cdot S$, both of which are isomorphic to the non-trivial simple $\mathbb{Z}_2$-representation. Therefore $G' = G$ and $R_1 = V$ is not simple. In case $R_1$ is not simple, one considers the irreducible component which contains $S$ and introduces a (possibly) bigger decomposition group $H_{2,1}''$. Hence, after reducing to the Clifford setting and assuming $K$ to be algebraically closed, the embedding problem (\theref{classic}) gives a one-to-one correspondence between the embedding of $S$ into an irreducible $H_{2,1}''$-module and factor sets 
$$\alpha: H_{2,1}''/H_1 \times H_{2,1}''/H_1 \to \mathbb{C}$$
and modules of finite degree of the algebra $\mathfrak{a} = \mathbb{C}[H_{2,1}''/H_1]$ corresponding to $\alpha^{-1}$, i.e. in which the multiplication is given by
$$ \ov{h_1} \ov{h_2} = \alpha^{-1}(\ov{h_1},\ov{h_2})\ov{h_1h_2}.$$
 
So at stage $i$, we decomposed $M_{d-i}$ into 
$$KH_ia + KH_i(S^1_1 \oplus \cdots \oplus S^1_{e_1}) + \cdots + KH_i(S^{i-1}_1 \oplus \cdots \oplus S^{i-1}_{e_{i-1}}) \oplus S^i_1 \oplus \cdots \oplus S^i_{e_i},$$
and the $S^i_j$ ($1 \leq j \leq e_i$) give rise to decomposition groups $H_i \subset H''_{i+1,j} \subset H_{i+1}$ and embeddings $S^i_j \subset V^i_j$, where the latter is a simple $H''_{i+1,j}$-module and generates a part of $KH_{i+1}S^i_j$.\\

However, we are not yet satisfied by our decomposition of $M^H$. In fact, we would like a more subtle relation between the $H_i$'s and $G_i$'s, as we already remarked above. Let's reconsider the decomposition of $M_{d-1}^H$ into simple $H_1$-modules 
$$ M^H_{d-1} = S^0_1 \oplus \cdots \oplus S^0_{e_0} \oplus S^1_1 \oplus \cdots \oplus S^1_{e_1},$$
in which the $S^0_i$ add up to $KH_1a$. Some of these $S^0_i$ are (simple) $G_1$-modules. The others give rise to a $KH_1 \subseteq KG_1$-fragment $M_{d-1} \supseteq S^0_i$. In any case, since $M_{d-1} = KG_1a$ and by the irreducibility of $M$ we know that every $S^1_j$ must be of the form $gS^0_{j_i}$ for some $g \in G_1 \setminus H_1$. Indeed, otherwise we would have a non-trivial $FKG$-subfragment by killing such an $S^1_j$. Therefore, the simple $S^0_j$ determine $M_{d-1}$. If $S^0_j$ is a $G_1$-module, there is nothing to it. If not, then 
$$S^0_j \oplus g_2S^0_j \oplus \cdots \oplus g_{m_j}S^0_j$$
for some $g_2,\ldots, g_{m_j} \in G_1 \setminus H_1$ appears in the decomposition of $M_{d-1}$. Therefore, assume that $S^0_1, \ldots, S^0_{f_0}$ are the building blocks of $M_{d-1}^H$, by which we mean that these $S^0_i$'s contain all the $G_1$-modules and exactly enough $H_1$-modules such that all its conjugates relative to $G_1$ yield the decomposition of $M_{d-1}^H$. If one then fixes an $S = S^0_j$, one can introduce a decomposition group $G_1' \subset G_1$ like in the classical theory. Note that by our construction, every building block gives rise to a (possibly different) decomposition group! Moreover, if $S$ happens to be a $G_1$-module, then $G_1' = G_1$.\\

Now, we investigate what happens if we look at $M_{d-2} = KG_2a$. Recall that we considered $KH_2S$. If this $H_2$-module is even a $G_2$-module, then all the conjugates of $S$ relative to $G_1$ become equal inside $M_{d-2}^H$:
$$KH_2gS = gKH_2S = KH_2S.$$
If $KH_2S$ is not a $G_2$-module, then $S$ and $gS$ remain conjugated. By definition, there exist elements $h_2,\ldots, h_m \in H_2 \setminus H_1$ such that
$$KH_2S = S \oplus h_2S \oplus \cdots \oplus h_mS.$$
For any $g \in G_1$ we get
$$KH_2gS = gKH_2S = g(S \oplus h_2S \oplus \cdots \oplus h_mS).$$
We easily calculate
$$KH_1gh_iS = gK(g^{-1}H_1g)h_iS = gh_i K(h_i^{-1}H_1h_i)S = gh_iKH_1S,$$
since $ H_1 \triangleleft G_1$ and $H_1 \triangleleft H_2$. Hence, $gh_iS$ is a simple $H_1$-module and 
$$ KH_2gS = gS \oplus gh_2S \oplus \cdots \oplus gh_mS.$$
We also deduce that the simple $H_1$-factors of $KH_2S$ and $KH_2gS$ are either all the same, or all different. So in total, we obtain that all the conjugates of $S$ relative to $G_1$ contribute to the decomposition of $KG_2S$ into simple $H_1$-modules. This also explains that two different building blocks $S$ and $S'$ don't affect one another at a higher stage.\\

Now, if $KG_1S = S \oplus g_2S \oplus \cdots \oplus g_lS$ as $H_1$-modules for some $g_2,\ldots, g_l \in G_1 \setminus H_1$, then by the above we obtain that
$$KG_2S = KH_2S \oplus g_2KH_2S \oplus \cdots \oplus  g_nKH_2S \oplus r_1KH_2S \oplus \cdots \oplus r_tKH_2S,$$
for $n \leq l$ (up to some possible reordering) and $r_1,\ldots, r_t \in G_2 \setminus G_1H_2$.
Hence we again can define $G_2' \subset G_2$ to be the subgroup of elements that leave the sum of all simples equivalent to $KH_2S$ invariant. 
From our discussion above, if $ S \cong g_iS$ as $H_1$-modules, then $KH_2S$ and  $g_iKH_2S$ have the same decomposition into simple $H_1$-components. However, this does not imply that $KH_2S$ and $g_iKH_2$ are isomorphic as $H_2$-modules! A disappointing side effect is that there is no chance at all that $G_1' \subset G_2'$. Observe moreover, that if our normal chain of subgroups is maximal, then $G_2 = G_1H_2$ and no $r$'s appear (see below).\\

Before we step up the ladder one stair further, we establish the foregoing for a concrete example.

\begin{example}
Look at the following graph of groups
$$\begin{array}{ccc}
\mathbb{Z}_4^j = \{1,j,-1,-j\} & \triangleleft& Q_8 = <-1,i,j | i^2 = j^2 =  -1, ij = -ji >\\
\triangledown & & \triangledown\\
\mathbb{Z}_2 = \{1,-1\} & \triangleleft & \mathbb{Z}_4^i = \{1,i,-1,-i\}
\end{array}$$
As transversal sets, we choose $T_1 = T_2 = \{1, j\}$ and one checks that the associated 2-cocycle $h$ takes values in $H_1 = \mathbb{Z}_2$. This implies that $h$ is filtered. 
We know that $Q_8$ has four 1-dimensional representations, given by
$$\begin{array}{c}
T_1 : i \mapsto 1,~ j \mapsto 1 \\
T_2 : i \mapsto -1,~ j \mapsto 1 \\
T_3 : i \mapsto 1,~ j \mapsto -1 \\
T_4 : i \mapsto - 1,~ j \mapsto - 1 
\end{array}$$
and one simple 2-dimensional representation
$$ U: i \mapsto \begin{pmatrix} i & 0 \\ 0 & -i \end{pmatrix},~ j \mapsto \begin{pmatrix} 0 & - 1 \\ 1 & 0 \end{pmatrix}$$
If we consider $U$ as a $\mathbb{Z}_4^j$-representation, we diagonalise
$$ \frac{1}{2}i \begin{pmatrix} - i & - 1 \\ - i  & 1 \end{pmatrix}\begin{pmatrix} 0 & - 1 \\ 1 & 0 \end{pmatrix}\begin{pmatrix} 1 &  1 \\ i & - i \end{pmatrix} = \begin{pmatrix} -i & 0\\ 0 & i \end{pmatrix}$$
and under base change 
$$\begin{array}{c}
e_1 = f_1 + if_2\\
e_2 = f_1 - i f_2
\end{array}$$
we get $U = V^{-i} \oplus V^i$, where $V^{i}$ is the simple $\mathbb{Z}_4$-representation, defined by $ j \mapsto i$ and similarly for $V^{-j}$. Consider now the $F\mathbb{C}Q_8$-fragment
$$ M = U \oplus T_3 \oplus T_2 \supseteq V^{-i} \oplus T_3 \oplus T_2 \supseteq \Delta$$
in which $\Delta$ stands for the diagonal of $M_1$. Let $\{f_1,f_2,e_3,e_4\}$ be a basis for $M_0$ establishing the direct sum decomposition. Then we will work with $\{e_1,e_3,e_4\}$ and $\{e_1+e_3+e_4\}$ as bases for $M_1$, resp. $M_2$. One convinces oneself that this is an irreducible fragment. We calculate
$$\mathbb{C}\mathbb{Z}_2\Delta = \mathbb{C}(e_3+e_4) \oplus \mathbb{C}e_1$$
as $\mathbb{Z}_2$-modules, and $\{\mathbb{C}e_1, \mathbb{C}(e_3+e_4)\}$ is a minimal set of building blocks. Furthermore
$$ M_1 = \mathbb{C}\mathbb{Z}_2\Delta \oplus j \cdot \mathbb{C}(e_3 + e_4),$$
in which the last component equals $\mathbb{C}(-e_3 + e_4)$. The first building block $S_1 = \mathbb{C}e_1$ is a $\mathbb{Z}_4^j$-module, so $G_1' = G_1$. For the second building block $S_2 = \mathbb{C}(e_3 + e_4)$, we have $S_2 \cong jS_2$, whence $G_1' = G_1 = \mathbb{Z}_4^j$ as well. Subsequently, we have
$$KH_2S_1 = \mathbb{C}\mathbb{Z}_4^ie_1 = \mathbb{C}e_1 \oplus i \cdot \mathbb{C}e_1 = \mathbb{C}e_1 \oplus \mathbb{C}e_2 = U,$$
and
$$ KH_2S_2 =  \mathbb{C}\mathbb{Z}_4^i(e_3 + e_4) = \mathbb{C}(e_3 + e_4) \oplus i \cdot \mathbb{C}(e_3 + e_4).$$
Since $jKH_2S_1 = KH_2S_1$ and $jKH_2S_2 = KH_2S_2$, we have that
$$ \mathbb{C}Q_8e_1 = \mathbb{C}\mathbb{Z}_4^ie_1 = U,~ \mathbb{C}Q_8(e_3 + e_4) = \mathbb{C}\mathbb{Z}_4^i(e_3 + e_4) = T_3 \oplus T_4.$$
Hence in both cases $G_2' = Q_8$.
\end{example}
Now, if $d > 2$, then we would have to look at $KH_3S$, but one can no longer apply the same techniques, since we do not know whether $H_1$ is normal in $H_3$. However, for every building block one can decompose $KH_2S$ into simple $H_2$-modules 
$$ KH_2S = T_1 \oplus \cdots \oplus T_n,$$
extend to a decomposition of $KG_2S$, choose a new set of building blocks and repeat the foregoing argument. So at every stage, the relation between $G_i, H_i, G_{i+1}$ and $H_{i+1}$ comes into play. Indeed, if $T$ is a simple $H_i$-rep, then you decompose 
$$\begin{array}{ll}
KG_iT = T \oplus g_2T \oplus \cdots \oplus g_mT & g_2,\ldots,g_m \in G_{i} \setminus H_i\\
KH_{i+1}T = T \oplus h_2T \oplus \cdots \oplus h_tT & h_2,\ldots, h_t \in H_{i+1}\setminus H_i\\
KG_{i+1}T = KH_{i+1}T \oplus g_2KH_{i+1}T \oplus \cdots  & \\
\quad \oplus g_nKH_{i+1}T \oplus r_1KH_{i+1}T \oplus \cdots \oplus r_uKH_{i+1}T &
\end{array}$$
for some $n \leq m$ and $r_2,\ldots,r_u \in G_{i+1} \setminus G_iH_{i+1}$ and we obtain a decomposition group $H_{i+1} \subseteq G'_{i+1,T} \subseteq G_{i+1}$. Unfortunately, one has no hope that an ascending chain of decomposition groups arises. In the previous example, we would have
$$ KH_2S_2 = \mathbb{C}\mathbb{Z}_4^iS_2 = T_3 \oplus T_2$$
as decomposition into simple $\mathbb{Z}_4^i$-modules, but $S_2 = \mathbb{C}(e_3 + e_4)$ does not fits nicely in one of the two simple components.\\

We summarize the foregoing in
\begin{theorem}
Let $K$ be a field of characteristic zero and let $H \triangleleft G$ be a normal subgroup of some finite group $G$, with fixed ascending chain of normal subgroups
$$1 \triangleleft G_1 \triangleleft \cdots \triangleleft G_d =G.$$
If $M$ is an irreducible $FKG$-fragment of essential length $d$, then $M_H$ is either an irreducible $FKH$-fragment or decomposes into a fragment direct sum 
$$M_H = \bigoplus_S  M(S),$$
where $ M(S) = KGS \supseteq KG_{d-1}S \supseteq \cdots \supseteq KG_1S \supset M_{d} \cap KG_1S \supseteq 0 \cdots$, and 
where the sum runs over a set of building blocks for $KG_1M_d$, consisting of $H_1$-modules. Moreover, to every $S$, we have a decomposition group $H_1 \subseteq G_{1,S}^{'} \subseteq G_1$.\\
In the latter case, to every such $S$ there is associated a set of building blocks of $KG_2S$ consisting of $H_2$-modules and we obtain a fragment direct sum 
$$M(S) = \bigoplus_T M(T),$$
where 
$$M(T) = KGT \supseteq KG_{d-1}T \supseteq \cdots \supseteq KG_2T \supseteq KG_1S \cap KG_2T \supseteq M_d  \cap KG_2T \supseteq 0 \cdots$$
Again the sum runs over the set of building blocks and we obtain decomposition groups $H_2 \subseteq G_{2,T}^{'} \subseteq G_2$. We obtain similar fragment decompostions of the $M(T)$ and so we arrive at decomposition groups $H_i \subseteq G_i' \subseteq G_i$ at every stage $1 \leq i \leq d$.
\end{theorem}

\begin{example}
We recover the classical Clifford theory for a normal subgroup $H \triangleleft G$, if we consider the trivial filtration
$$\begin{array}{ccc}
K &\subset& KG\\
 \veq && \triangledown\\
 K & \subset &KH.
 \end{array}$$
Indeed, suppose that $M \supseteq M_1$ is an irreducible $FKG$-fragment, with $M$ a simple $G$-module. We have that $M_1 = Ka$ is one-dimensional and $M = KGa$. If $M \supseteq M_1$ is also irreducible as $FKH$-fragment, then $M = KHa$ and it follows that $M$ is a simple $H$-module. If the $FKH$-fragment is not irreducible, we have that $KHa \neq KGa = M$. According to our approach, we decompose $KHa$ into simple $H$-modules $KHa = S_1 \oplus \cdots \oplus S_n$. Then we deduced that $M = KGa$ can be decomposed as a sum of simple $H$-modules which are all conjugate relative to G to one of the $S_i$. A set of building blocks that contains some $S_i$ and $S_j$ would entail that these simple $H$-modules are not conjugate relative to $G$ to one another. But then we would have that $KGS \subset M$ would be a proper $G$-submodule, a contradiction. Therefore, $\{S_1\}$ is a set of building blocks and we find that $M$ can be decomposed as 
$$M = S_1 \oplus g_2S_1 \oplus \cdots \oplus g_mS_1$$
for some $g_2,\ldots, g_m \in G \setminus H$.
\end{example}
 
\section{Geometric aspect of decomposition groups}

Suppose that the normal chain of subgroups is maximal. Observe that an irreducible $FKG$-fragment is completely determined by a $KG$-module $M = M_0$ and a one-dimensional $K$-subspace $Ka \subset M$. However, there are some constraints on the element $a \in M$; suppose that
\begin{equation}\label{M}
M = S_1^{n_1} \oplus \cdots \oplus S_k^{n_k} \oplus T_1^{m_1} \oplus \cdots \oplus T_l^{m_l},
\end{equation}
is a decomposition of $M$ into simple $G$-modules and $(e_i)_{i \in I}$ an ordered basis establishing this decomposition. Moreover, we assume that all the $T_j$ are 1 dimensional and the $S_i$ are $s_i$-dimensional with $s_i > 1$. Let $a = \sum a_ie_i$. If $Ke_i = T_j$, one of the one dimensional $G$-representations occurring in \eqref{M}, then since $M$ must equal $KGa$, the coefficient $a_i$ of $a$ is nonzero. For every $S_i$-module, we obtain a point $[a_0:a_1: \cdots : a_{s_i -1}] \in \mathbb{P}^{s_i-1}$ and it is clear that the choice of $a$ is indeed independent up to a scalar multiplication of the coefficients per simple component of $M$. In total, we obtain a point in the projective variety
 $$\mathbb{P}^{s_1-1} \times \cdots \times \mathbb{P}^{s_1-1} \times \mathbb{P}^{s_2-1} \times \cdots \times \mathbb{P}^{s_2-1} \times \cdots \times \mathbb{P}^{s_k-1} \times \cdots \times \mathbb{P}^{s_k-1},$$
 with $n_1$ factors $\mathbb{P}^{s_1-1}$, etc. \\

Since $KH_iT = KG_iT$ for all $T$ and $ 0 \leq i \leq d$, we don't get any non-trivial decomposition groups from the one dimensional simples. Concerning the higher dimensional simples, we fix $S = S_1$, which is $s = s_1$-dimensional and the part $[a_0: \cdots : a_{s-1}]$ of $a$. We can represent $KH_1a$ as an $ s \times |H_1|$-matrix $A_1$, of which the $i$-th column gives the action of $h_i$ on $a$ in the basis $\{e_0,\ldots,e_{s-1}\}$ of $S$. Similary, we introduce a matrix $B_1$ for $KG_1a$. Clearly, if $\rk(A_1) = \rk(B_1)$, then $KG_1a = KH_1a$ and $G_1' = G_1$ follows. Hence we obtain a Zariski open or closed set $X \subset \mathbb{P}^{s-1}$, on which no non-trivial decomposition groups occur. In our example above, for $S = U$ is 2-dimensional with basis $\{f_1,f_2\}$ and $a = a_0f_1 + a_1f_2$, we calculate (for ordering $H_1 = \mathbb{Z}_2 = \{1,-1\}$ and $G_1 = \mathbb{Z}_4^j = \{1,j,-1,-j\})$
$$A_1 = \begin{pmatrix} a_0 & -a_0  \\  a_1 & -a_1  \end{pmatrix},~ B_1 = \begin{pmatrix} a_0 & -a_1 & -a_0 & a_1 \\ a_1 & a_0 & -a_1 & -a_0 \end{pmatrix}$$
The rank of $A_1$ is always 1, so $X = \mathbb{V}(x_0^2 + x_1^2)$ if $\mathbb{P} = \Proj \mathbb{C}[x_0,x_1]$, which consists of two points $\{[1:i], [1:-i]\}$. Remark that these points correspond to the base change which establishes the decomposition of $U$ into two simple $\mathbb{Z}_4^j$-modules. We considered the point $[1:i]$ and obtained indeed that $G_1' = G_1$.\\

For $H_2$ and $G_2$ we introduce the (extended) matrices $A_2$ and $B_2$ as well as the Zariski open $X_2$. Now on $X_1 = X$, we have $KH_1a = KG_1a$, whence
$$KH_2a = KH_2KH_1a = KH_2KG_1a = KG_2a.$$
The last equality follows since our normal chain is maximal. Indeed, since $h_2g_1h_2'g_1' = h_2h_2'h_2'^{-1}g_1h_2'g_1' \in H_2G_1$ and $(hg)^{-1} = g^{-1}h^{-1} = h^{-1}hg^{-1}h^{-1} \in H_2G_1$, we see that $H_2G_1$ is a group and since 
$$g_2^{-1}h_2g_1g_2 = h_2h_2^{-1}g_2^{-1}h_2g_2g_2^{-1}g_1g_2 \in H_2G_1,$$
we have that $ G_1 \subset H_2G_1 \triangleleft G_2$, whence $H_2G_1 = G_2$. Therefore, $X_1 \subset X_2 \subset \mathbb{P}^{s-1}$. In the example
$$ A_2 = \begin{pmatrix} a_0 & ia_0 & -a_0 & -ia_0 \\ a_1 & -ia_1 & -a_1 & ia_1 \end{pmatrix},$$
which has rank 2 on $\mathbb{X}(x_0) \cap \mathbb{X}(x_1) = \mathbb{X}(x_0x_1)$. So the possible interesting cases are \newline$[a_0:a_1] \in \mathbb{P} \setminus \mathbb{X}(x_0x_1) = \{ [1:0],[0:1]\}$. Take $a = [1:0]$, the other case being analogous. We have
$$M_1 = \mathbb{C}\mathbb{Z}_4^jf_1 = \mathbb{C}f_1 \oplus j\cdot \mathbb{C}f_1 = \mathbb{C}f_1 \oplus \mathbb{C}f_2.$$
Both $\mathbb{Z}_2$-modules are equivalent, so $G_1' = G_1$. One stair further, we have
$$ \mathbb{C}Q_8 f_1 = \mathbb{C}f_1 \oplus j \cdot \mathbb{C}f_1 = \mathbb{C}f_1 \oplus \mathbb{C}f_2,$$
but both components are no longer isomorphic as $\mathbb{Z}_4^i$-modules. Therefore, $G_2' = \mathbb{Z}^i_4$. Observe that $G_1' \not\subset G_2'$!\\

\begin{example}
Look at the following graph of groups
$$\begin{array}{ccc}
<a> & \triangleleft & D_8 = <a,x | a^4 = x^2 =  1,~ xax^{-1} = a^{-1} >\\
\triangledown & & \triangledown\\
\{e, a^2\} & \triangleleft &  \{e,x,a^2,a^2x\}
\end{array}$$
and consider the two-dimensional simple representation $S$ defined by
$$a \mapsto \begin{pmatrix} i & 0 \\ 0 & - i \end{pmatrix},~ x \mapsto \begin{pmatrix} 0 & 1 \\ 1 & 0 \end{pmatrix}.$$
\end{example}
Then one calculates that $X_1 = \mathbb{V}(x_0x_1) = \{ [1:0], [0:1]\} \subset X_2 = \mathbb{X}(x_0^2 - x_1^2)$.

\end{document}